\documentclass[preprint,3p]{elsarticle}

\usepackage[ansinew]{inputenc}
\usepackage{amsmath, amsfonts} 
\usepackage{amsthm, amssymb}
\usepackage[english]{babel}
\usepackage[T1]{fontenc}
\usepackage{lmodern}
\usepackage{graphicx}
\usepackage{subfigure}

\newtheorem{theorem}{Theorem}[section]

\newtheorem{corollary}{Corollary}[section]

\newtheorem{assumption}{Assumption}

\theoremstyle{definition}
\newtheorem{definition}{Definition}
\newtheorem{remark}{Remark}[section]

\usepackage{array}
\usepackage{float}
\usepackage{subfigure}
\usepackage{color}

\newcommand{\guillemets}[1]{``#1''}

\journal{Systems and Control Letters}

\begin{document}

\begin{frontmatter}



\title{Contraction of monotone phase-coupled oscillators}


\author[Alex]{A. Mauroy}
\author[Rodolphe]{R. Sepulchre}

\address[Alex]{Department of Mechanical Engineering, University of California Santa Barbara, Santa Barbara, CA 93106, USA. E-mail address: alex.mauroy@engr.ucsb.edu}
\address[Rodolphe]{Department of Electrical Engineering and Computer Science, University of Liège, B-4000 Liège, Belgium}

\begin{abstract}
This paper establishes a global contraction property for networks of phase-coupled oscillators characterized by a monotone coupling function. The contraction measure is a total variation distance. The contraction property determines the asymptotic behavior of the network, which is either finite-time synchronization or asymptotic convergence to a splay state.
\end{abstract}


\end{frontmatter}


\section{Introduction}

Networks of coupled oscillators are a general paradigm to understand a wealth of natural phenomena \cite{Buck,Peskin,Sherman} as well as an efficient model for the design of engineered systems \cite{Macdonald,Paley}. In the limit of weak coupling, realistic models of limit-cycle oscillators (evolving in high-dimensional spaces) can be reduced to one-dimensional models of phase oscillators (evolving on the circle) coupled through their phase differences \cite{Kuramoto_book,Winfree}. Though more amenable to mathematical analysis, these generic phase-coupled models may exhibit rich and complex ensemble behaviors and have attracted intense research interest in the past decades (e.g. Kuramoto model \cite{Kuramoto3}).

A phase-coupled model is characterized by its \textsl{coupling function} (which is closely related to the \textsl{phase response curve} of the oscillators). The seminal work of Kuramoto \cite{Kuramoto_book} assumes a sinusoidal coupling function. The present paper rather considers a network of oscillators characterized by a monotone coupling function (we denote these oscillators as \textsl{monotone oscillators}). This model was first studied in \cite{Kuramoto3}, through the phase reduction of the popular leaky integrate-and-fire neuron model. Using local analysis, the author investigated the effect of a delay on the network stability and showed that the periodic collective motion of the oscillators becomes quasiperiodic when the delay exceeds a critical value. Since it is known that delayed oscillators lose their monotonicity property, this result suggests a link between the stability properties of the oscillators and their monotonicity properties.

In contrast to the \emph{local} results provided in \cite{Kuramoto3}, we present in this paper a \emph{global} contraction property for networks of monotone oscillators. The result is shown with respect to a $1$-norm which has the interpretation of a total variation distance and which is inspired from our previous studies \cite{Mauroy3,Mauroy}. The contraction property of the model determines the asymptotic behavior of the network. Namely, monotone oscillators either achieve perfect synchronization in finite time or converge to a unique anti-synchronized state (\textsl{splay configuration}).

The paper is organized as follows. Section \ref{sec_monot_osci} introduces the model of monotone phase-coupled oscillators. In Section \ref{sec_contract}, we present our main result on the contraction property of the oscillators. The result is exploited in Section \ref{sec_behavior} to study the collective behaviors of the network. Finally, the paper closes with some concluding remarks in Section \ref{sec_conclu}.

\section{Monotone phase-coupled oscillators}
\label{sec_monot_osci}

We consider a network of identical \textsl{phase-coupled oscillators}. A single (uncoupled) phase oscillator $k$ is characterized by a phase $\theta_k \in \mathbb{S}^1(0,2\pi)$ that evolves on the circle with constant velocity $\dot{\theta}_k=\omega$, where $\omega$ is the natural frequency of the oscillator. Within the network, $N$ phase oscillators are (all-to-all) coupled through their phase differences: they evolve on the $N$-torus $\mathbb{T}^N=\mathbb{S}^1\times \cdots \times \mathbb{S}^1$ according to the canonical phase dynamics
\begin{equation}
\label{gen_form_Kuramoto}
\dot{\theta}_k=\omega+\sum_{\substack{j=1\\j\neq k}}^N \Gamma(\theta_k-\theta_j) \qquad k=1,\dots,N\,.
\end{equation}

Using averaging techniques, every network of weakly-coupled identical limit-cycle oscillators can be reduced to the form \eqref{gen_form_Kuramoto}, with an appropriate coupling function $\Gamma(\cdot)$ that is closely related to the phase response curve of the oscillators \cite{Hoppensteadt,Kuramoto_book}. The phase response curve characterizes the phase sensitivity of the oscillator to a (infinitesimal) perturbation, a quantity that can be either numerically computed or experimentally measured. Since it is a function computed on the (periodic) limit cycle, the phase response curve is $2\pi$-periodic, and so is the coupling function $\Gamma(\cdot)$.\\

We assume that the coupling function satisfies the following monotonicity assumption, in which case the phase-coupled oscillators are called \guillemets{\textsl{monotone oscillators}}.
\begin{assumption}[Monotonicity]
\label{assump_monotone}
The coupling function $\Gamma(\cdot)$ is strictly monotone, i.e. either $\Gamma'(\theta)>0$ $\forall \theta \in (0,2\pi)$ or $\Gamma'(\theta)<0$ $\forall \theta \in (0,2\pi)$, with $\Gamma'$ denoting the first derivative of $\Gamma$.
\end{assumption}
\begin{remark} Assumption \ref{assump_monotone} implies that the coupling function is discontinuous, that is $\Gamma(0^-)\neq \Gamma(0^+)$. However, it is important to note that the $2\pi$-periodicity condition imposes $\Gamma(0)=\Gamma(2\pi)$ (Figure \ref{fig_discont_coupl}).
\end{remark}
\begin{figure}[h]
\begin{center}
\includegraphics[width=7cm]{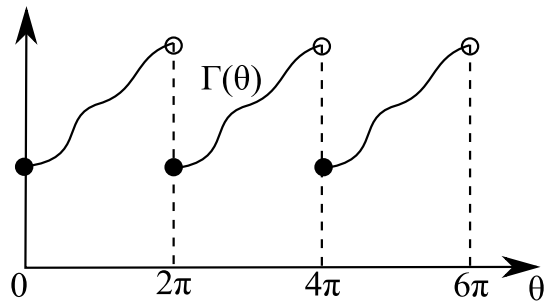}
\caption{The coupling function of monotone oscillators is discontinuous but satisfies the $2\pi$-periodicity condition, so that $\Gamma(0)=\Gamma(2\pi)$.}
\label{fig_discont_coupl}
\end{center}
\end{figure}

The monotonicity property is frequently encountered in spiking oscillators, at least in good approximation. Monotone phase-coupled oscillators were first obtained in the case of (weakly) pulse-coupled leaky integrate-and-fire oscillators \cite{Kuramoto3}. More generally, for a purely impulsive coupling, monotone oscillators appear as the phase reduction of oscillators characterized by a monotone phase response curve, such as van der Pol oscillators with strong relaxation or limit-cycle oscillators near a homoclinic bifurcation \cite{Brown}.\\

\begin{remark}
The analysis of a network of monotone oscillators is well-known when all oscillators are initialized within a \textsl{semicircle}: in this case, a change of coordinates maps the dynamics into $\mathbb{R}^N$ and the model falls into the well-studied category of consensus models on a convex set, see e.g. \cite{Moreau,Sarlette_book_chapter}. In the present paper, the emphasis is on the whole (non-convex) $N$-torus $\mathbb{T}^N$.
\end{remark}

\section{Main result}
\label{sec_contract}

\begin{definition}[Contraction] Let $\mathcal{S}$ be a continuous-time dynamical system defined on a metric space $X$ (with the distance $d$) and let $\mathbf{\Phi}$ denote the flow associated with $\mathcal{S}$, i.e. $\mathbf{\Phi}(\mathbf{x},t)$ is an orbit of $\mathcal{S}$, with $\mathbf{\Phi}(\mathbf{x},0)=\mathbf{x}\in X$. Then $\mathcal{S}$ is contracting (resp. expanding) in $X_0\subseteq X$ with respect to $d$ if
\begin{equation*}
\frac{d}{dt}d\left(\mathbf{\Phi}(\mathbf{x},t),\mathbf{\Phi}(\mathbf{y},t)\right)<0 \quad \textrm{(resp. }>0\textrm{)}
\end{equation*}
for all $\mathbf{x}\neq \mathbf{y}$ and for all $t$ such that $\mathbf{\Phi}(\mathbf{x},t)\in X_0$, $\mathbf{\Phi}(\mathbf{y},t)\in X_0$.
\end{definition}

We will study the contraction of model \eqref{gen_form_Kuramoto} with respect to the distance
\begin{equation}
\label{1_norm}
d(\mathbf{x},\mathbf{y})=\left\|\mathbf{x}-\mathbf{y}\right\|_{(1)}=|x_1-y_1|+\sum_{j=1}^{n-1} |x_{j+1}-y_{j+1}-(x_j-y_j)|+|x_n-y_n|\,.
\end{equation}
This particular $1$-norm distance is motivated by previous results on integrate-and-fire models. For finite populations of leaky integrate-and-fire oscillators, the distance \eqref{1_norm} was successfully used in \cite{Mauroy} to capture the contraction property of the so-called \textsl{firing map} (a discrete-time map that provides snapshots of the network configuration at the successive firings of the oscillators \cite{Mirollo}). For infinite populations, the continuous equivalent of \eqref{1_norm} induces a Lyapunov function for the partial derivative equation that governs the evolution of the population density \cite{Mauroy3}.

The distance has the remarkable interpretation of a total variation distance. Indeed, \eqref{1_norm} corresponds to the total variation of a piecewise linear function that interpolates the values $x_i-y_i$ (Figure \ref{tvd}). In the continuous case (related to infinite populations), this interpretation still holds since the continuous equivalent of \eqref{1_norm} is the $L_1$-norm of the derivative \cite{Dunford}.
\begin{figure}[h]
\begin{center}
\includegraphics[width=5cm]{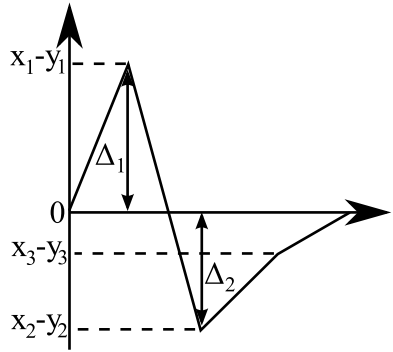}
\caption{The distance \eqref{1_norm} is interpreted as the total variation of the linear interpolation of the values $x_i-y_i$. In the example of the figure, the distance is given by $\left\|\mathbf{x}-\mathbf{y}\right\|_{(1)}=2\Delta_1+2\Delta_2$, a quantity which corresponds to the total variation of the piecewise linear function.
}
\label{tvd}
\end{center}
\end{figure}

Before establishing our main result, we must notice that the dynamics \eqref{gen_form_Kuramoto} are invariant with respect to a rigid rotation of the oscillators, a property which prevents a contraction in the full space $\mathbb{T}^N$. We therefore remove this marginally stable rigid mode and consider the equivalent dynamics expressed in a rotating frame associated with an oscillator (without loss of generality, we choose oscillator $1$). Denoting the phase differences by $\tilde{\theta}_k=\theta_{k+1}-\theta_1 \in [0,2\pi]$ for $k=1,\dots,N-1$, the dynamics \eqref{gen_form_Kuramoto} are rewritten as the $(N-1)$-dimensional dynamics
\begin{equation}
\label{phase_dyn_modified}
\dot{\tilde{\theta}}_k=\Gamma(\tilde{\theta}_k)+\sum_{\substack{j=1\\j\neq k}}^{N-1} \Gamma(\tilde{\theta}_k-\tilde{\theta}_j)- \sum_{j=1}^{N-1} \Gamma(-\tilde{\theta}_j)\,.
\end{equation}
Without loss of generality, we assume that the oscillators satisfy the phase ordering $\tilde{\theta}_k \leq \tilde{\theta}_{k+1}$, which does not change over time since the oscillators are identical. Solutions $\mathbf{\tilde{\Theta}}(t)=(\tilde{\theta}_1(t),\dots,\tilde{\theta}_{N-1}(t))$ of \eqref{phase_dyn_modified} do not evolve in $\mathbb{T}^{N-1}$, but only in the closure $\bar{\mathcal{C}}$ of the open cone $\mathcal{C}=\{\mathbf{\tilde{\Theta}}\in (0,2\pi)^{N-1}|\tilde{\theta}_k< \tilde{\theta}_{k+1}, k=1\dots, N-2\}$. An orbit that reaches the boundary of $\mathcal{C}$ corresponds to the synchronization of at least two oscillators.\\

Now, under a mild technical assumption on the curvature of the coupling function, we are in position to prove the contraction property of the dynamics \eqref{phase_dyn_modified} with respect to the $1$-norm distance \eqref{1_norm}. The result is summarized in the following theorem.
\begin{theorem}
\label{theo_contract_consensus}
Consider a coupling function that satisfies (i) Assumption \ref{assump_monotone} (monotonicity) and (ii) either $\Gamma''(\theta)\geq 0$ $\forall \theta \in (0,2\pi)$ or $\Gamma''(\theta)\leq 0$ $\forall \theta \in (0,2\pi)$. Then, \eqref{phase_dyn_modified} is either contracting (if $\Gamma'<0$) or expanding (if $\Gamma'>0$) in $\mathcal{C}$ with respect to \eqref{1_norm}.

\end{theorem}
\begin{proof}

We prove that the dynamics are contracting when $\Gamma'<0$. Without loss of generality, we suppose that $\Gamma''\leq 0$ and, considering the two orbits $\mathbf{\tilde{\Theta}}=(\tilde{\theta}_1,\dots,\tilde{\theta}_{N-1})$ and \mbox{$\mathbf{\tilde{\Psi}}=(\tilde{\psi}_1,\dots,\tilde{\psi}_{N-1})$}, we assume that $\tilde{\theta}_1-\tilde{\psi}_1\leq 0$. The proof of the other cases follows on similar lines.

We consider the distance \eqref{1_norm} between $\mathbf{\tilde{\Theta}}$ and $\mathbf{\tilde{\Psi}}$ and remove the absolute values to obtain
\begin{equation}
\label{equa_dist_1_chap7}
\left\|\mathbf{\tilde{\Theta}}-\mathbf{\tilde{\Psi}}\right\|_{(1)}=\sum_{k=1}^{N_c} (-1)^k \left(\tilde{\theta}_{\mathcal{K}(k)}-\tilde{\psi}_{\mathcal{K}(k)}\right)\,.
\end{equation}
The subscripts $\mathcal{K}(k)$ correspond to $N_c$ \textsl{critical oscillators} ($1\leq N_c \leq N-1$) whose phases are characterized by a change of the sign of $\tilde{\theta}_{j+1}-\tilde{\psi}_{j+1}-(\tilde{\theta}_j-\tilde{\psi}_j)$. Formally, the map \mbox{$\mathcal{K}:\{1,\dots,N_c\}\mapsto \{1,\dots,N-1\}$} is defined so that $\mathcal{K}(k-1)<\mathcal{K}(k)$ and so that
\begin{equation*}
\left[\tilde{\theta}_{\mathcal{K}(k)+1}-\tilde{\psi}_{\mathcal{K}(k)+1}-\left(\tilde{\theta}_{\mathcal{K}(k)}-\tilde{\psi}_{\mathcal{K}(k)}\right)\right]\left[\tilde{\theta}_{\mathcal{K}(k)}-\tilde{\psi}_{\mathcal{K}(k)}-\left(\tilde{\theta}_{\mathcal{K}(k)-1}-\tilde{\psi}_{\mathcal{K}(k)-1}\right)\right]< 0\,,
\end{equation*}
with $\tilde{\theta}_0-\tilde{\psi}_0=\tilde{\theta}_N-\tilde{\psi}_N=0$. The index $\mathcal{K}(k)$ corresponds to a maximum of $\tilde{\theta}_j-\tilde{\psi}_j$ when $k$ is even and to a minimum when $k$ is odd (Figure \ref{fig_critical_indices}). In particular, one has
\begin{equation}
\label{rel_kappa}
(-1)^k \left(\tilde{\theta}_{\mathcal{K}(k)}-\tilde{\psi}_{\mathcal{K}(k)}\right) > (-1)^k \left(\tilde{\theta}_{\mathcal{K}(k\pm 1)}-\tilde{\psi}_{\mathcal{K}(k\pm 1)}\right) \,,
\end{equation}
assuming that $\tilde{\theta}_{\mathcal{K}(0)}-\tilde{\psi}_{\mathcal{K}(0)}=\tilde{\theta}_{\mathcal{K}(N_c+1)}-\tilde{\psi}_{\mathcal{K}(N_c+1)}=0$.
\begin{figure}[h]
\begin{center}
\subfigure[$k$ even]{\includegraphics[width=4cm]{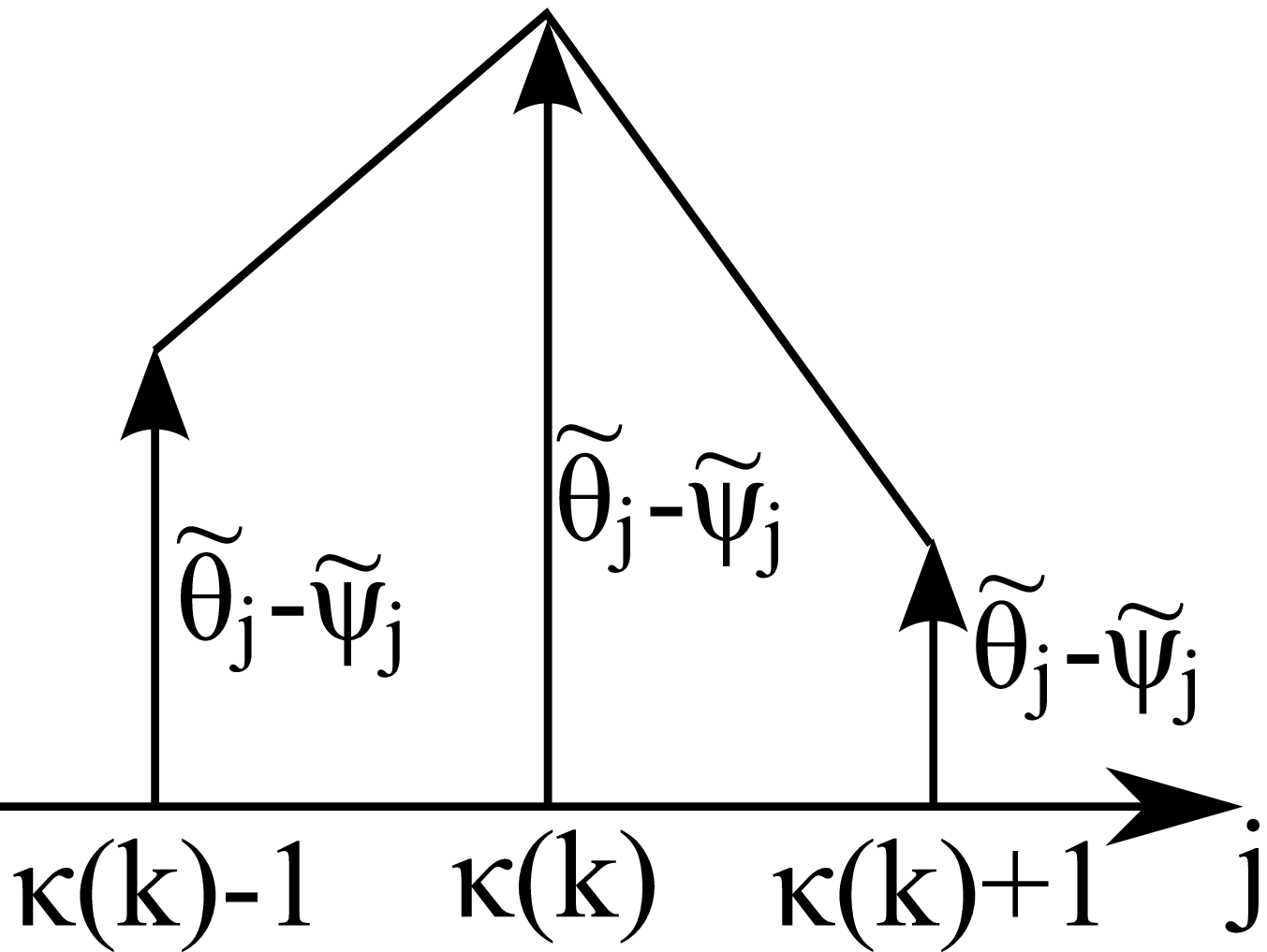}}
\hspace{1cm}
\subfigure[$k$ odd]{\includegraphics[width=4cm]{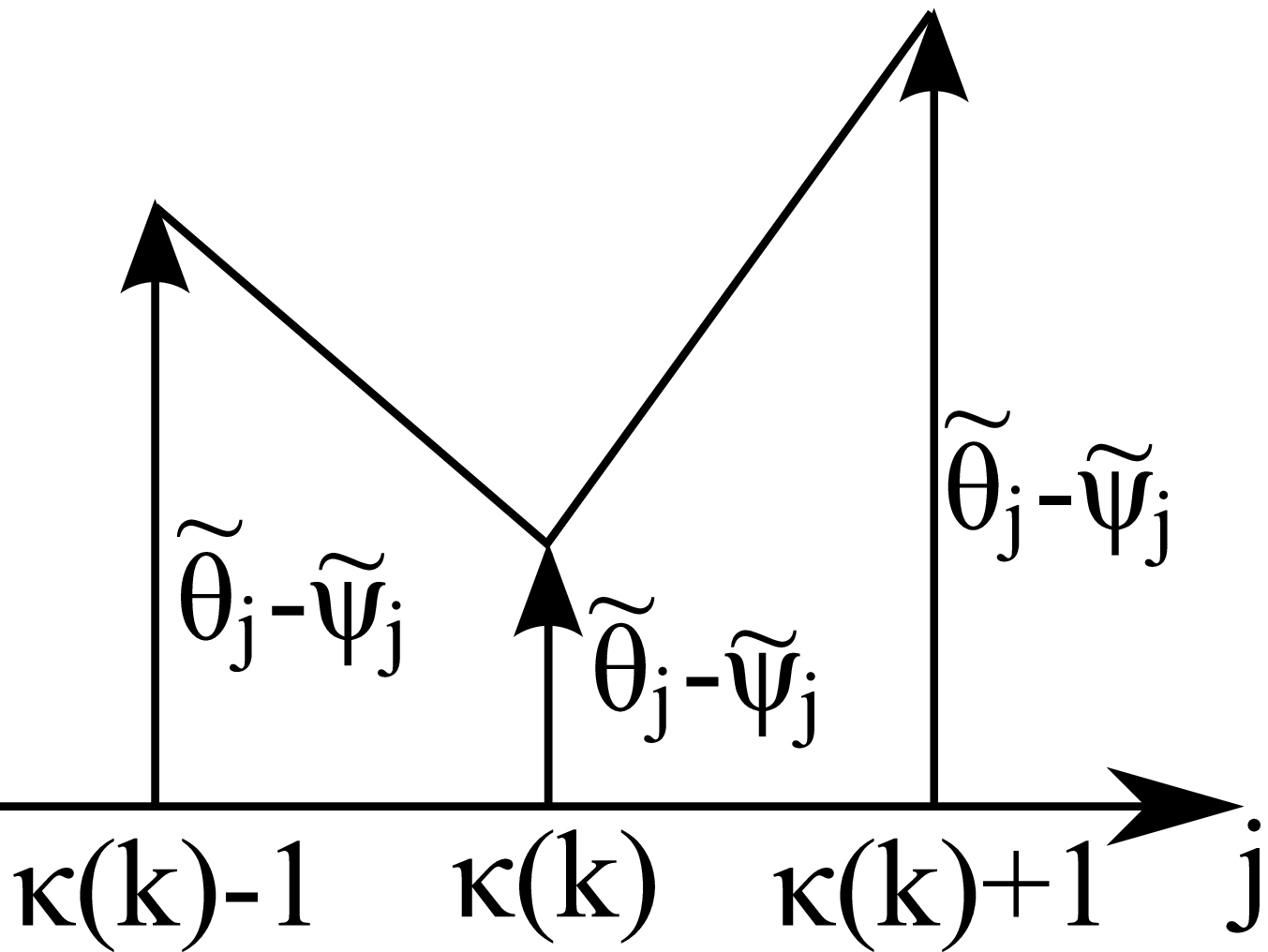}}
\caption{The critical oscillators correspond to extremal values of $\tilde{\theta}_j-\tilde{\psi}_j$.}
\label{fig_critical_indices}
\end{center}
\end{figure}

The time derivative of \eqref{equa_dist_1_chap7} yields
\begin{equation}
\label{deriv_terms_k}
\frac{d}{dt}\left\|\mathbf{\tilde{\Theta}}-\mathbf{\tilde{\Psi}}\right\|_{(1)}=\sum_{k=1}^{N_c} (-1)^k \left(\dot{\tilde{\theta}}_{\mathcal{K}(k)}-\dot{\tilde{\psi}}_{\mathcal{K}(k)}\right) \triangleq \sum_{k=1}^{N_c} T^{(k)}\,.
\end{equation}
Next, we consider separately each term $T^{(k)}$. It follows from \eqref{phase_dyn_modified} that
\begin{equation*}
\begin{split}
T^{(k)}=(-1)^k\, \Bigg[\Gamma(\tilde{\theta}_{\mathcal{K}(k)})-\Gamma(\tilde{\psi}_{\mathcal{K}(k)}) & +\sum_{\substack{j=1\\j\neq \mathcal{K}(k)}}^{N-1} \Gamma(\tilde{\theta}_{\mathcal{K}(k)}-\tilde{\theta}_j)-\Gamma(\tilde{\psi}_{\mathcal{K}(k)}-\tilde{\psi}_j)\\
& -\sum_{j=1}^{N-1} \Gamma(-\tilde{\theta}_j)-\Gamma(-\tilde{\psi}_j) \Bigg]\,.
\end{split}
\end{equation*}
Using the mean value theorem, we obtain
\begin{equation}
\label{T_k_tot}
\begin{split}
T^{(k)}&=\underbrace{\Gamma'(\xi_{\mathcal{K}(k)}) (-1)^k \left(\tilde{\theta}_{\mathcal{K}(k)}-\tilde{\psi}_{\mathcal{K}(k)}\right)}_{\displaystyle \triangleq T^{(k)}_{\mathcal{K}(k)}}\\
&+ \sum_{\substack{j=1\\j\neq \mathcal{K}(k)}}^{N-1}  \underbrace{\Gamma'(\xi_{\mathcal{K}(k),j}) (-1)^k  \left(\tilde{\theta}_{\mathcal{K}(k)}-\tilde{\psi}_{\mathcal{K}(k)}-(\tilde{\theta}_j-\tilde{\psi}_j)\right)}_{\displaystyle \triangleq T^{(k)}_j} \\
& +(-1)^k \underbrace{\left(-\sum_{j=1}^{N-1} \Gamma(-\tilde{\theta}_j)-\Gamma(-\tilde{\psi}_j)\right)}_{\displaystyle \triangleq T_\Sigma} \,,
\end{split}
\end{equation}
with $\xi_{\mathcal{K}(k)}\in [\tilde{\theta}_{\mathcal{K}(k)},\tilde{\psi}_{\mathcal{K}(k)}]$ or $ \in [\tilde{\psi}_{\mathcal{K}(k)},\tilde{\theta}_{\mathcal{K}(k)}]$ and $\xi_{\mathcal{K}(k),j}\in [(\tilde{\theta}_{\mathcal{K}(k)}-\tilde{\theta}_j) \bmod 2\pi,(\tilde{\psi}_{\mathcal{K}(k)}-\tilde{\psi}_j) \bmod 2\pi]$ or $\in [(\tilde{\psi}_{\mathcal{K}(k)}-\tilde{\psi}_j) \bmod 2\pi,(\tilde{\theta}_{\mathcal{K}(k)}-\tilde{\theta}_j) \bmod 2\pi]$. Since $\Gamma''\leq0$, the values satisfy
\begin{eqnarray}
\label{rel_xi_1}
\xi_{\mathcal{K}(k)} \leq \xi_{\mathcal{K}(k+1)} &&\qquad k<N_c\,,\\
\label{rel_xi_2}
\xi_{\mathcal{K}(k),j} \leq \xi_{\mathcal{K}(k+1),j} &&\qquad k<N_c\,, \quad j<\mathcal{K}(k) \textrm{ or } j>\mathcal{K}(k+1) \,,\\
\label{rel_xi_3}
\xi_{\mathcal{K}(N_c),j} \leq \xi_{\mathcal{K}(1),j} &&\qquad \mathcal{K}(1)<j<\mathcal{K}(N_c)\,.
\end{eqnarray}
Because of the discontinuity of $\Gamma$, these relationships are not valid for an orbit that lies in the boundary of $\mathcal{C}$.\\

Disregarding the terms $T_\Sigma$, we show that each positive term $T_j^{(k)}$ of $T^{(k)}$ (in \eqref{T_k_tot}) can be associated with a term $T_j^{(k+1)}$ of $T^{(k+1)}$ so that the addition of both terms $T_j^{(k)}+T_j^{(k+1)}$ is negative.
\begin{itemize}
\item If the term $T_{\mathcal{K}(k)}^{(k)}\geq 0$, it is associated with the term $T_{\mathcal{K}(k+1)}^{(k+1)}$. It follows from \eqref{rel_xi_1} and from $\Gamma''\leq 0$ that $\Gamma'(\xi_{\mathcal{K}(k+1)})\leq \Gamma'(\xi_{\mathcal{K}(k)})<0$. In addition, \eqref{rel_kappa} implies
\begin{equation*}
-(-1)^{k+1} \left(\tilde{\theta}_{\mathcal{K}(k+ 1)}-\tilde{\psi}_{\mathcal{K}(k+ 1)} \right) < (-1)^k \left(\tilde{\theta}_{\mathcal{K}(k)}-\tilde{\psi}_{\mathcal{K}(k)} \right) \leq 0\,.
\end{equation*}
Then, we obtain $T_{\mathcal{K}(k)}^{(k)}+T_{\mathcal{K}(k+1)}^{(k+1)}<0$.
\item If the term $T_j^{(k)}\geq 0$ for any $j\neq \mathcal{K}(k)$, it is associated with the term $T_j^{(k+1)}$, $j\neq \mathcal{K}(k+1)$. (According to \eqref{rel_kappa} and the definition of the critical oscillators, one has never $T_j^{(k)}\geq 0$ for $\mathcal{K}(k)<j\leq\mathcal{K}(k+1)$.) It follows from \eqref{rel_xi_2} and from $\Gamma''\leq 0$ that $\Gamma'(\xi_{\mathcal{K}(k+1),j})\leq \Gamma'(\xi_{\mathcal{K}(k),j})<0$. In addition, \eqref{rel_kappa} implies 
\begin{equation*}
-(-1)^{k+1} \left(\tilde{\theta}_{\mathcal{K}(k+ 1)}-\tilde{\psi}_{\mathcal{K}(k+ 1)}-(\tilde{\theta}_j-\tilde{\psi}_j) \right) < (-1)^k \left(\tilde{\theta}_{\mathcal{K}(k)}-\tilde{\psi}_{\mathcal{K}(k)}-(\tilde{\theta}_j-\tilde{\psi}_j) \right) \leq 0 \,.
\end{equation*}
Then, we obtain $T_j^{(k)}+T_j^{(k+1)}<0$.
\end{itemize}

For $k=N_c$, the terms of $T^{(k)}$ cannot be associated with the terms of $T^{(k+1)}$. Then, it remains to consider the terms $T_j^{(N_c)}$ and the terms $(-1)^k T_\Sigma$. We distinguish two cases: the case $N_c$ even and the case $N_c$ odd.

\textbf{Case $\mathbf{N_c}$ even.} The addition of the $N_c$ terms $(-1)^k T_\Sigma$ yields $\sum_{k=1}^{N_c} (-1)^k T_\Sigma=0$.
In addition, the relationship \eqref{rel_kappa} implies 
\begin{equation}
\label{rel_N_c}
\tilde{\theta}_{\mathcal{K}(1)}-\tilde{\psi}_{\mathcal{K}(1)} < 0 \qquad \tilde{\theta}_{\mathcal{K}(N_c)}-\tilde{\psi}_{\mathcal{K}(N_c)} >0
\end{equation}
so that the term $T^{(N_c)}_{\mathcal{K}(N_c)}< 0$.
 
If the term $T_j^{(N_c)}\geq 0$ for any $j\neq \mathcal{K}(N_c)$, it is associated with the term $T_j^{(1)}$, $j\neq \mathcal{K}(1)$. (According to \eqref{rel_N_c} and the definition of the critical oscillators, one has never $T_j^{(N_c)}\geq 0$ for $j\leq \mathcal{K}(1)$ or for $j \geq \mathcal{K}(N_c)$.) It follows from \eqref{rel_xi_3} and from $\Gamma''\leq 0$ that $\Gamma'(\xi_{\mathcal{K}(1),j})\leq \Gamma'(\xi_{\mathcal{K}(N_c),j})<0$. In addition, the inequalities \eqref{rel_N_c} imply
\begin{equation*}
-(-1)^1 \left(\tilde{\theta}_{\mathcal{K}(1)}-\tilde{\psi}_{\mathcal{K}(1)}-(\tilde{\theta}_j-\tilde{\psi}_j) \right) < (-1)^{N_c} \left(\tilde{\theta}_{\mathcal{K}(N_c)}-\tilde{\psi}_{\mathcal{K}(N_c)}-(\tilde{\theta}_j-\tilde{\psi}_j) \right) \leq 0\,.
\end{equation*}
Then, we obtain $T_j^{(N_c)}+T_j^{(1)}<0$.

\textbf{Case $\mathbf{N_c}$ odd.} The addition of the $N_c$ terms $(-1)^k T_\Sigma$ yields
\begin{equation*}
\sum_{k=1}^{N_c} (-1)^k T_\Sigma=-T_\Sigma=-\sum_{j=1}^{N-1} \Gamma'(\xi_{-j}) \left(\tilde{\theta}_j-\tilde{\psi}_j\right)\,
\end{equation*}
with $\xi_{-j} \in [2\pi-\tilde{\theta}_j,2\pi-\tilde{\psi}_j]$ or $\in [2\pi-\tilde{\psi}_j,2\pi-\tilde{\theta}_j]$. Since $\Gamma''\leq 0$, the values satisfy
\begin{eqnarray}
\xi_{\mathcal{K}(N_c),j}\leq \xi_{-j} && \qquad j<\mathcal{K}(N_c)   \,, \label{rel_xi_4_a} \\
\xi_{-j} \leq \xi_{\mathcal{K}(1),j} && \qquad j>\mathcal{K}(1) \,. \label{rel_xi_4_b}
\end{eqnarray}

In addition, the relationship \eqref{rel_kappa} implies
\begin{equation}
\label{rel_N_c2}
\tilde{\theta}_{\mathcal{K}(1)}-\tilde{\psi}_{\mathcal{K}(1)} < 0 \qquad \tilde{\theta}_{\mathcal{K}(N_c)}-\tilde{\psi}_{\mathcal{K}(N_c)} <0
\end{equation}
so that the term $T^{(N_c)}_{\mathcal{K}(N_c)}< 0$.

\begin{itemize}
\item If the term $T_j^{(N_c)}\geq 0$ for any $j\neq \mathcal{K}(N_c)$, it is associated with the term $-\Gamma'(\xi_{-j}) (\tilde{\theta}_j-\tilde{\psi}_j)$. (According to \eqref{rel_N_c2} and the definition of the critical oscillators, one has never $T_j^{(N_c)}\geq 0$ for $j\geq \mathcal{K}(N_c)$.) It follows from \eqref{rel_xi_4_a} and from $\Gamma'' \leq 0$ that $\Gamma'(\xi_{-j})\leq \Gamma'(\xi_{\mathcal{K}(N_c),j})<0$. In addition, \eqref{rel_N_c2} implies
\begin{equation*}
\tilde{\theta}_j-\tilde{\psi}_j < (-1)^{N_c} \left(\tilde{\theta}_{\mathcal{K}(N_c)}-\tilde{\psi}_{\mathcal{K}(N_c)}-(\tilde{\theta}_j-\tilde{\psi}_j) \right) \leq 0\,.
\end{equation*}
Then, we obtain $T_j^{(N_c)}-\Gamma'(\xi_{-j}) (\tilde{\theta}_j-\tilde{\psi}_j)<0$.

\item If the term $-\Gamma'(\xi_{-j}) (\tilde{\theta}_j-\tilde{\psi}_j)\geq 0$, it is associated with the term $T_j^{(1)}$, $j\neq \mathcal{K}(1)$. (According to \eqref{rel_N_c2} and the definition of the critical oscillators, one has never $-\Gamma'(\xi_{-j}) (\tilde{\theta}_j-\tilde{\psi}_j)\geq 0$ for $j\leq \mathcal{K}(1)$.) It follows from \eqref{rel_xi_4_b} and from $\Gamma''\leq 0$ that $\Gamma'(\xi_{\mathcal{K}(1),j})\leq \Gamma'(\xi_{-j})<0$. In addition, \eqref{rel_N_c2}
implies
\begin{equation*}
-(-1)^1 \left(\tilde{\theta}_{\mathcal{K}(1)}-\tilde{\psi}_{\mathcal{K}(1)}-(\tilde{\theta}_j-\tilde{\psi}_j) \right) < -(\tilde{\theta}_j-\tilde{\psi}_j) \leq 0\,.
\end{equation*}
Then, we obtain $-\Gamma'(\xi_{-j}) (\tilde{\theta}_j-\tilde{\psi}_j)+T_j^{(1)}<0$.\\
\end{itemize}

Finally, every term of \eqref{deriv_terms_k}-\eqref{T_k_tot} is either negative or can be associated with a unique term so that the addition of both is negative. Thus, one has
\begin{equation*}
\frac{d}{dt}\left\|\mathbf{\tilde{\Theta}}-\mathbf{\tilde{\Psi}}\right\|_{(1)}=\sum_{k=1}^{N_c} T^{(k)} <0\,.
\end{equation*}

For the other situations ($\Gamma''\geq 0$, $\Gamma'>0$), the proof follows on similar lines (except that the terms of $T^{(k)}$ might be associated with the terms of $T^{(k-1)}$ instead).

\end{proof}

Theorem \ref{theo_contract_consensus} proves the contraction property inside the open cone $\mathcal{C}$, but not for an orbit that lies in the boundary of $\mathcal{C}$. However, the contraction property also holds for each of the $(N-2)$-dimensional cones $\bar{\mathcal{C}}_i'= \{\mathbf{\tilde{\Theta}}\in\bar{\mathcal{C}}|\tilde{\theta}_{i-1}=\tilde{\theta}_{i}\}$, $i=1,\dots,N$ (with $\tilde{\theta}_0=0$ and $\tilde{\theta}_N=2\pi$), that partition the boundary of $\mathcal{C}$. A cone $\bar{\mathcal{C}}'_i$ corresponds to the synchronization of two oscillators and is therefore invariant under the dynamics \eqref{phase_dyn_modified} ---~since two synchronized oscillators remain synchronized forever. The result on the contraction of \eqref{phase_dyn_modified} in the open cone $\mathcal{C}'_i$ ---~defined as the interior of $\bar{\mathcal{C}}_i'$~--- is summarized in the following corollary.

\begin{corollary}
\label{corol}
Under the same assumptions as Theorem \ref{theo_contract_consensus}, the dynamics \eqref{phase_dyn_modified} defined in a cone $\bar{\mathcal{C}}'_i$ ($i=1,\dots,N$) are either contracting (if $\Gamma'<0$) or expanding (if $\Gamma'>0$) in the open cone $\mathcal{C}'_i$ with respect to \eqref{1_norm}.
\end{corollary}
\begin{proof}
The proof is a  straightforward corollary of the proof of Theorem \ref{theo_contract_consensus}.\\
\textbf{In the cone $\mathbf{\mathcal{C}'_1}$.} Since $\tilde{\theta}_1=0$, the dynamics \eqref{phase_dyn_modified} are reduced to the $(N-2)$-dimensional dynamics
\begin{equation*}
\dot{\tilde{\theta}}_k=2\Gamma(\tilde{\theta}_k)+\sum_{\substack{j=2\\j\neq k}}^{N-1} \Gamma(\tilde{\theta}_k-\tilde{\theta}_j)- \sum_{j=2}^{N-1} \Gamma(-\tilde{\theta}_j)-\Gamma(0) \quad k=2,\dots,N-1\,,
\end{equation*}
which are similar to the dynamics \eqref{phase_dyn_modified} in dimension $N-2$. Then, the proof follows on similar lines as the proof of Theorem \ref{theo_contract_consensus}. The only difference is that the term $T^{(k)}_{\mathcal{K}(k)}$ in \eqref{T_k_tot} is replaced by $2\, T^{(k)}_{\mathcal{K}(k)}$, but this modification does not affect the validity of the proof. (In the proof, each term $2\, T^{(k)}_{\mathcal{K}(k)}$ is exclusively associated with another term $2\,T^{(k+1)}_{\mathcal{K}(k+1)}$.)\\
\textbf{In the cone $\mathbf{\mathcal{C}'_i}$, $\mathbf{i=2,\dots,N-1}$.} Since $\tilde{\theta}_i=\tilde{\theta}_{i-1}$, the dynamics \eqref{phase_dyn_modified} are reduced to the $(N-2)$-dimensional dynamics
\begin{equation*}
\begin{cases}
\displaystyle
\dot{\tilde{\theta}}_k=\Gamma(\tilde{\theta}_k)+\hspace{-0.4cm} \sum_{\substack{j=1\\j\neq \{k,i-1,i\}}}^{N-1} \hspace{-0.4cm} \Gamma(\tilde{\theta}_k-\tilde{\theta}_j)+2\Gamma(\tilde{\theta}_k-\tilde{\theta}_i)- \hspace{-0.4cm}\sum_{\substack{j=1\\j\neq \{i-1,i\}}}^{N-2} \hspace{-0.4cm} \Gamma(-\tilde{\theta}_j)-2\Gamma(-\tilde{\theta}_i) & k\neq \{i-1,i\}\\\\
\displaystyle
\dot{\tilde{\theta}}_i=\Gamma(\tilde{\theta}_i)+\sum_{\substack{j=1\\j\neq \{i-1,i\}}}^{N-1} \Gamma(\tilde{\theta}_i-\tilde{\theta}_j)+\Gamma(0) - \sum_{\substack{j=1\\j\neq \{i-1,i\}}}^{N-2} \Gamma(-\tilde{\theta}_j)-2\Gamma(-\tilde{\theta}_i)  &
\end{cases}
\end{equation*}
and the proof follows on similar lines as the proof of Theorem \ref{theo_contract_consensus}. The terms $T_i^{(k)}$ in \eqref{T_k_tot}, with $\mathcal{K}(k) \neq i$, are replaced by $2\,T_i^{(k)}$ and the term $\Gamma'(\xi_{-i})(\theta_i-\psi_i)$ of $T_\Sigma$ is replaced by $2\Gamma'(\xi_{-i})(\theta_i-\psi_i)$. The reader can easily verify that these modifications do not affect the validity of the proof. (In the proof, the above-mentioned terms are exclusively combined between themselves.)\\
\textbf{In the cone $\mathbf{\mathcal{C}'_N}$.} The situation is similar to the case of the cone $\mathcal{C}'_1$.
\end{proof}

In addition, an equivalent result also holds for every intersection of several cones $\bar{\mathcal{C}}'_i$, intersection that corresponds to the synchronization of three or more oscillators and that is therefore invariant under the dynamics \eqref{phase_dyn_modified}. The result is summarized in the following corollary. (The proof follows on similar lines as the proof of Corollary \ref{corol} and is not detailed in the present paper.)

\begin{corollary}
\label{corol2}
Let $I$ denote a non-empty subset of $\{1,\dots,N\}$. Under the same assumptions as Theorem \ref{theo_contract_consensus}, the dynamics \eqref{phase_dyn_modified} defined in the intersection $\bar{\mathcal{C}}'_I=\bigcap_{i\in I} \bar{\mathcal{C}}'_i$ are either contracting (if $\Gamma'<0$) or expanding (if $\Gamma'>0$) in the interior of $\bar{\mathcal{C}}'_I$ with respect to \eqref{1_norm}.
\end{corollary}

\section{Collective behaviors of monotone oscillators}
\label{sec_behavior}

Theorem \ref{theo_contract_consensus} provides a contraction property for networks of monotone oscillators. This property is now shown to determine the asymptotic behavior of the network.

\begin{theorem}
\label{theo_behav}
Consider a network of $N$ monotone phase-coupled oscillators with the dynamics \eqref{gen_form_Kuramoto}. Provided that the coupling function has a curvature of constant sign on $(0,2\pi)$, for almost every initial condition,
\begin{itemize}
\item the oscillators asymptotically converge to the incoherent splay state configuration $\tilde{\theta}^*_k=k\frac{2\pi}{N}$, with $k\in\{1,\dots,N-1\}$, if the coupling function is strictly decreasing on $(0,2\pi)$;
\item the oscillators achieve perfect synchronization $\theta^*_1(t)=\theta^*_2(t)=\dots=\theta^*_N(t)$ (i.e. $\tilde{\theta}^*_k\in\{0,2\pi\}$) in finite time if the coupling function is strictly increasing on $(0,2\pi)$.
\end{itemize}
\end{theorem}
\begin{proof}
\textbf{\textbf{Convergence to the splay state if $\mathbf{\Gamma'<0}$.}}
We first note that an orbit $\mathbf{\Phi}(\mathbf{\tilde{\Theta}},t)$ of \eqref{phase_dyn_modified} with an initial condition $\mathbf{\tilde{\Theta}}$ in the open cone $\mathcal{C}$ cannot reach the boundary of $\mathcal{C}$. Indeed, straightforward computations show the repelling property of the boundary:
\begin{equation}
\label{velocity_bound}
\begin{array}{ll}
\dot{\tilde{\theta}}_1 \approx \Gamma(0^+)-\Gamma(2\pi^-)>0 & \textrm{for } \tilde{\theta}_1\approx 0\,,\\
\dot{\tilde{\theta}}_{k+1}-\dot{\tilde{\theta}}_k \approx \Gamma(0^+)-\Gamma(2\pi^-)>0 & \textrm{for } \tilde{\theta}_k\approx \tilde{\theta}_{k+1}\,, \\
\dot{\tilde{\theta}}_{N-1} \approx \Gamma(2\pi^-)-\Gamma(0^+)<0 & \textrm{for } \tilde{\theta}_{N-1}\approx 2\pi\,.
\end{array}
\end{equation}
Then, Theorem \ref{theo_contract_consensus} implies that the dynamics \eqref{phase_dyn_modified} are contracting in $\mathcal{C}$ for all $t$ and, for any small fixed $\Delta t$, the discrete-time mapping $\mathbf{\tilde{\Theta}}\mapsto \mathbf{\Phi}(\mathbf{\tilde{\Theta}},\Delta t)$ is also contracting in $\mathcal{C}$. From the contraction mapping theorem, it follows that, for all initial conditions in $\mathcal{C}$ (the interior of the cone), the solutions of \eqref{phase_dyn_modified} asymptotically converge to the unique fixed point lying in $\mathcal{C}$. This unique fixed point corresponds to the splay state $\tilde{\theta}^*_k=k\frac{2\pi}{N}$, which concludes the first part of the proof.\\
\textbf{\textbf{Convergence to synchronization when $\mathbf{\Gamma'>0}$.}}
Consider an orbit $\mathbf{\Phi}(\mathbf{\tilde{\Theta}},t)$ of \eqref{phase_dyn_modified}, with $\mathbf{\tilde{\Theta}}\in \mathcal{C}$ that is not the fixed point $\tilde{\theta}^*_k=k\frac{2\pi}{N}$. Theorem \ref{theo_contract_consensus} implies that the distance between the orbit and the fixed point monotonically increases. As a consequence, the orbit converges toward the boundary of $\mathcal{C}$ (corresponding to the synchronization of at least two oscillators). In addition, \eqref{velocity_bound} shows that the orbit approaches the boundary with the finite velocity $|\Gamma(2\pi^-)-\Gamma(0^+)|$  and therefore reaches the boundary in finite time.

The orbit subsequently evolves within a $(N-2)$-dimensional invariant cone $\mathcal{C}_i'$. Given Corollary \ref{corol}, the reduced dynamics in $\mathcal{C}_i'$ satisfies the contraction property. It follows that any fixed point lying in $\mathcal{C}_i'$ is unstable, so that, for almost every initial condition, the orbit could not have reached $\mathcal{C}_i'$ through a fixed point. Next, the same argument as above proves that the orbit reaches in finite time the boundary of $\mathcal{C}'_i$, an event that corresponds to another finite-time synchronization of (at least) two oscillators.

The orbit subsequently evolves within a $(N-3)$-dimensional cone $\mathcal{C}_I'$. Given Corollary \ref{corol2}, the reduced dynamics in $\mathcal{C}_I'$ still satisfies the contraction property, so that the above argument can be repeated.

Finally, the argument is repeated as many times as a pairwise synchronization occurs and after (at most) $N-1$ successive pairwise synchronizations, the network achieves full synchronization in finite time.
\end{proof}

The two asymptotic collective behaviors of monotone oscillators are illustrated in Figure \ref{fig_proof_synchro}. Monotone phase-coupled oscillators obtained through the averaging of pulse-coupled leaky integrate-and-fire oscillators (see \cite{Kuramoto3}) satisfy the mild curvature assumption on $\Gamma$, so that the result applies.\\

\begin{figure}[h]
\begin{center}
\subfigure[Incoherent splay state]{\includegraphics[width=8cm]{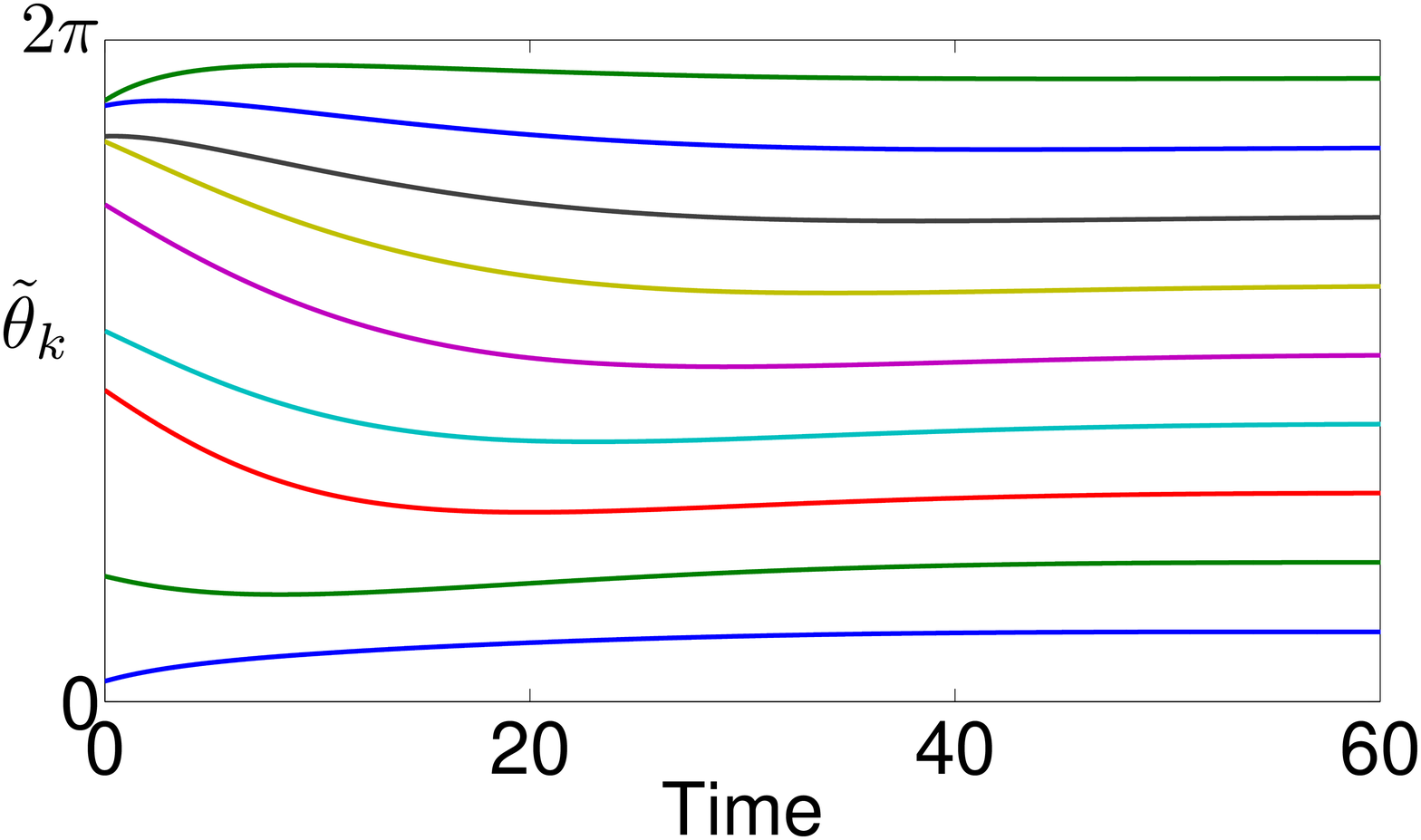}}
\subfigure[Synchronization]{\includegraphics[width=8cm]{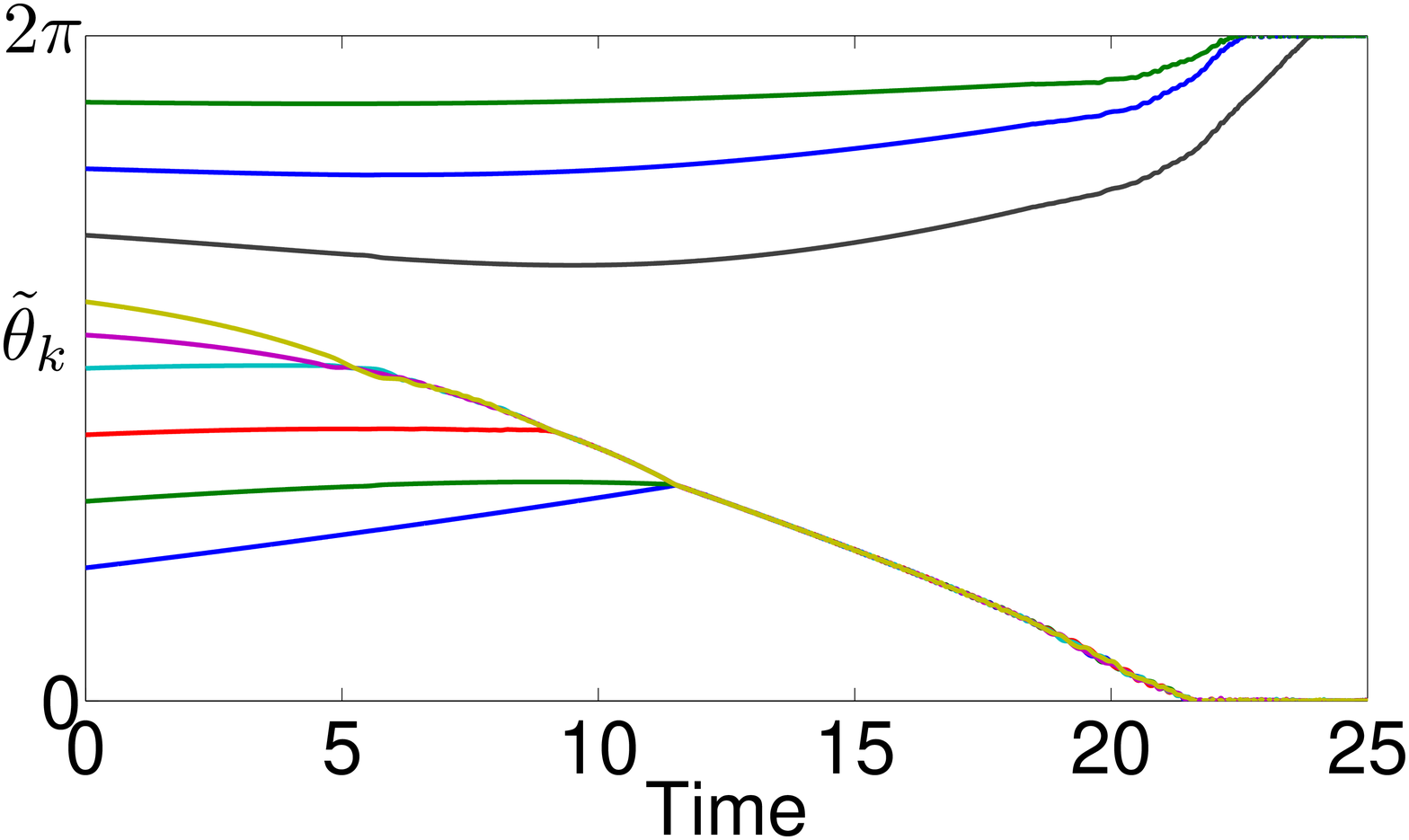}}
\caption{$N=10$ monotone phase-coupled oscillators exhibit two opposite asymptotic behaviors. (a) With the decreasing coupling function $\Gamma(\theta)=1/N(0.1+e^{-\theta})$, the oscillators asymptotically converge to the splay state. (b) With the increasing coupling function $\Gamma(\theta)=-1/N(0.1+e^{-\theta})$, several pairwise synchronizations occur, until the network achieves perfect synchronization in finite time.}
\label{fig_proof_synchro}
\end{center}
\end{figure}

Theorem \ref{theo_behav} is an illustration that contraction has strong implications on the asymptotic behavior. The global behavior of a network of monotone oscillators is in fact reminiscent of the global behavior of Kuramoto model, which is a gradient system. An important feature of the present model with respect to Kuramoto model is brought up by the discontinuity of the coupling function: the synchronization takes place in \emph{finite time}, and the splay state is an \emph{isolated fixed point}, whereas it is a $N-3$ dimensional manifold in Kuramoto model. A consequence of that difference is that the asymptotic behavior of \eqref{gen_form_Kuramoto} is robust to small heterogeneity in the natural frequencies $\omega$ in firing monotone oscillators \cite{Mauroy} whereas the asymptotic dynamics of Kuramoto model can be highly complex even for small heterogeneities \cite{Strogatz}.

\section{Conclusion}
\label{sec_conclu}

We have investigated the global stability properties of populations of monotone phase-coupled oscillators, thereby complementing local results presented in the earlier study \cite{Kuramoto3}. In particular, we showed that monotone oscillators only display two asymptotic collective behaviors: for almost all initial conditions, monotone oscillators either synchronize in finite time or asymptotically converge to the unique splay configuration.

The global stability analysis relies on a strong contraction property of the dynamics, which is the main result of this paper. Interestingly, the contraction is not captured with respect to a quadratic norm, but through a $1$-norm that has the interpretation of a total variation distance. In a general context, this result stresses the key role of $1$-norms on cones to connect the monotonicity of a system to its stability properties.

\section*{Acknowledgments}

This paper presents research results of the Belgian Network DYSCO (Dynamical Systems, Control, and Optimization), funded by the Interuniversity Attraction Poles Programme, initiated by the Belgian State, Science Policy Office. The scientific responsibility rests with its authors. A. Mauroy holds a postdoctoral fellowship from the Belgian American Educational Foundation.



\bibliographystyle{elsarticle-num}



\end{document}